\documentclass[12pt]{article}
\usepackage{amssymb}
\usepackage{mathrsfs}
\usepackage[hypertex]{hyperref}
\usepackage{amsfonts}
\usepackage{graphicx}
\usepackage{mathptmx}
\usepackage{latexsym,amsmath,amssymb,amsfonts,amsthm}

\newtheorem{theorem}{Theorem}[section]
\newtheorem{lemma}[theorem]{Lemma}

\newtheorem{corollary}[theorem]{Corollary}

\newtheorem{remark}[theorem]{Remark}

\newtheorem{definition}[theorem]{Definition}

\begin{document}
\setcounter{page}{1}
\title{Eigenvalue estimates for submanifolds in Hadamard manifolds and product manifolds $N\times\mathbb{R}$}
\author{Jing Mao$^{\ast}$, Rong-Qiang Tu and Kai Zeng}
\date{}
\protect\footnotetext{\!\!\!\!\!\!\!\!\!\!\!\!
{$^{\ast}$Corresponding author}
   }
\maketitle ~~~\\[-15mm]
\begin{center}{\footnotesize  Faculty of Mathematics and Statistics,\\
Key Laboratory of Applied Mathematics of Hubei Province,\\
 Hubei University, Wuhan, 430062, China \\
jiner120@163.com, jiner120@tom.com}
\end{center}

\begin{abstract}
In this paper, we investigate submanifolds with locally bounded mean
curvature in Hadamard manifolds, product manifolds
$N\times\mathbb{R}$, submanifolds with bounded $\varphi$-mean
curvature in the hyperbolic space, and successfully give lower
bounds for the weighted fundamental tone and the first eigenvalue of
the $p$-Laplacian.
\end{abstract}

\markright{\sl\hfill  J. Mao, R.-Q. Tu and K. Zeng \hfill}

\section{Introduction} \label{section1}
\renewcommand{\thesection}{\arabic{section}}
\renewcommand{\theequation}{\thesection.\arabic{equation}}
\setcounter{equation}{0} \setcounter{maintheorem}{0}

Let $(M,g)$ be an $n$-dimensional ($n\geq2$) smooth Riemannian
manifold with the Riemannian metric $g$, the gradient operator
$\nabla$ and the Laplacian $\Delta=\mathrm{div}\circ\nabla$. For an
open bounded connected domain $\Omega\subset M$, the classical
Dirichlet eigenvalue problem on $\Omega$ is actually to find
possible real numbers such that the following boundary value problem
(BVP for short)
\begin{eqnarray} \label{1-1}
\left\{
\begin{array}{ll}
\Delta{u}+\lambda{u}=0  \qquad\quad {\rm{in}}~\Omega,\\
 u=0   \qquad \qquad \quad ~~~ {\rm{on}}~\partial\Omega,    & \quad
\end{array}
\right.
\end{eqnarray}
has a nontrivial solution $u$. The desired real numbers $\lambda$
are called \emph{eigenvalues} of $\Delta$, and the space of
solutions of each $\lambda$ is called its eigenspace which is a
vector space. It is well known that for the BVP (\ref{1-1}), the
self-adjoint operator $\Delta$ only has the discrete spectrum whose
elements (i.e., eigenvalues) can be listed increasingly as follows
\begin{eqnarray*}
0<\lambda_{1}(\Omega)<\lambda_{2}(\Omega)\leq\cdots\uparrow\infty,
\end{eqnarray*}
and each associated eigenspace has finite dimension. $\lambda_{i}$
($i\geq1$) is called the $i$th Dirichlet eigenvalue of $\Delta$. By
\emph{Domain monotonicity of eigenvalues with vanishing Dirichlet
data} (cf. \cite[pp. 17-18]{ic}), we know that
$\lambda_{1}(\Omega_1)\leq\lambda_{1}(\Omega_2)$ if
$\Omega_1\supset\Omega_2$.

For a domain $\Omega\subseteq M$ (with or without boundary
$\partial\Omega$), one can define \emph{the fundamental tone}
$\lambda^{\ast}_{1}(\Omega)$ of $\Omega$ as
\begin{eqnarray*}
\lambda^{\ast}_{1}(\Omega):=\inf\left\{\frac{\int_{\Omega}\|\nabla
f\|^{2}dv}{\int_{\Omega}f^{2}dv}\Big{|}f\in
W^{1,2}_{0}(\Omega),f\neq0\right\},
\end{eqnarray*}
where $W^{1,2}_{0}(\Omega)$ is the completion of the set
$C^{\infty}_{0}(\Omega)$ of smooth functions compactly supported on
$\Omega$ under the Sobolev norm
$\|u\|_{1,2}=\{\int_{\Omega}(|u|^{2}+\|\nabla
u\|^{2})dv\}^{\frac{1}{2}}$, with $dv$ the Riemannian volume element
with respect to the metric $g$. In the sequel, without
specification, \emph{$\|\cdot\|$ denotes the norm of some prescribed
vector field}, and, for convenience, \emph{all the integrals will be
dropped the corresponding measures}. If $\Omega$ is unbounded, then
the fundamental tone $\lambda^{\ast}_{1}(\Omega)$ coincides with the
infimum $\inf(\Sigma)$ of the spectrum $\Sigma\subseteq[0,+\infty)$
of the unique self-adjoint extension of the Laplacian $\Delta$
acting on $C^{\infty}_{0}(\Omega)$, which is also denoted by
$\Delta$. If $\Omega$ has compact closure and piecewise smooth
boundary $\partial\Omega$ (maybe nonempty),
$\lambda^{\ast}_{1}(\Omega)$ equals the first closed eigenvalue (if
$\partial\Omega=\emptyset$) or the first Dirichlet eigenvalue (if
$\partial\Omega\neq\emptyset$) $\lambda_{1}(\Omega)$ of $\Delta$. If
$\Omega_{1}\subset\Omega_{2}$ are bounded domains, then
$\lambda^{*}_{1}(\Omega_{1})\geq\lambda^{*}_{1}(\Omega_{2})\geq0$.

From the above introduction, we know that \emph{for a bounded domain
$\Omega$ with boundary, the degree of smoothness of the boundary
$\partial\Omega$ decides the fundamental domain
$\lambda^{\ast}_{1}(\Omega)$ would degenerate into the first
Dirichlet eigenvalue $\lambda_{1}(\Omega)$ of the Laplacian or not}.

  Let $B_{M}(q,\ell)$ be a geodesic ball, with
center $q$ and radius $\ell$, on a complete noncompact Riemannian
manifold $M$. By the monotonicity of the first Dirichlet eigenvalue
$\lambda_{1}$ or the fundamental tone $\lambda_{1}^{\ast}$, one can
define a limit $\lambda_{1}(M)$ by
\begin{eqnarray*}
\lambda_{1}(M):=\lim\limits_{\ell\rightarrow\infty}\lambda_{1}\left(B_{M}(q,\ell)\right)=\lim\limits_{\ell\rightarrow\infty}\lambda_{1}^{\ast}\left(B_{M}(q,\ell)\right),
\end{eqnarray*}
which is independent of the choice of the center $q$. Clearly,
$\lambda_{1}(M)\geq0$. Schoen and Yau \cite[p. 106]{sy} suggested
that it is an important question to find conditions which will imply
$\lambda_{1}(M)>0$. Speaking in other words, manifolds with
$\lambda_{1}(M)>0$ might have some special geometric properties.
There are many interesting results supporting this. For instance,
 Mckean \cite{hpm}
showed that for an $n$-dimensional complete noncompact, simply
connected Riemannian manifold $M$ with sectional curvature
$K_{M}\leq-a^{2}<0$, $\lambda_{1}(M)\geq\frac{(n-1)^{2}a^{2}}{4}>0$,
and moreover,
$\lambda_{1}(\mathbb{H}^{n}(-a^{2}))=\frac{(n-1)^{2}a^{2}}{4}$ with
$\mathbb{H}^{n}(-a^{2})$ the $n$-dimensional hyperbolic space of
sectional curvature $-a^{2}$. Grigor'yan \cite{ag} showed that if
$\lambda_{1}(M)>0$, then $M$ is non-parabolic. Cheung and Leung
\cite{cl} proved that if $M$ is an $n$-dimensional complete minimal
submanifold in the hyperbolic $m$-space $\mathbb{H}^{m}(-1)$, then
$\lambda_{1}(M)\geq\frac{(n-1)^2}{4}>0$, and moreover, $M$ is
non-parabolic, i.e., there exists a non-constant bounded subharmonic
function on $M$. They also showed that if furthermore $M$ has at
least two ends, then there exists on $M$ a non-constant bounded
harmonic function with finite Dirichlet energy.

Consider the BVP
\begin{eqnarray} \label{1-2}
\left\{
\begin{array}{ll}
\Delta_{\varphi}{u}+\lambda{u}=0  \qquad\quad {\rm{in}}~\Omega,\\
 u=0   \qquad \qquad \quad ~~~ ~~{\rm{on}}~\partial\Omega,    & \quad
\end{array}
\right.
\end{eqnarray}
where $\Omega\subset M$ is an open bounded connected domain in a
given Riemannian manifold $M$, $\Delta_{\varphi}u:=\Delta
u-\langle\nabla u,\nabla\varphi\rangle$ is the weighted Laplacian
(also called the drifting Laplacian) on $M$, and $\varphi$ is a
real-valued smooth function on $M$. Similar to the BVP (\ref{1-1}),
$\Delta_{\varphi}$ in the BVP (\ref{1-2}) only has the discrete
spectrum and all the eigenvalues in the discrete spectrum can be
listed increasingly. By Rayleigh's theorem and Max-min principle, it
is easy to know that the first Dirichlet eigenvalue
$\lambda_{1,\varphi}(\Omega)$ of $\Delta_{\varphi}$ on $\Omega$ can
be characterized by
\begin{eqnarray*}
\lambda_{1,\varphi}(\Omega)=\inf\left\{\frac{\int_{\Omega}\|\nabla
f\|^{2}e^{-\varphi}}{\int_{\Omega}f^{2}e^{-\varphi}}\Big{|}f\in
W^{1,2}_{0}(\Omega),f\neq0\right\}.
\end{eqnarray*}
Similar to the case of the Laplacian, for a (bounded or unbounded)
domain $\Omega\subseteq M$ (with or without boundary
$\partial\Omega$), one can define \emph{the weighted fundamental
tone} $\lambda^{\ast}_{1,\varphi}(\Omega)$ of $\Omega$ as
\begin{eqnarray*}
\lambda^{\ast}_{1,\varphi}(\Omega):=\inf\left\{\frac{\int_{\Omega}\|\nabla
f\|^{2}e^{-\varphi}}{\int_{\Omega}f^{2}e^{-\varphi}}\Big{|}f\in
W^{1,2}_{0}(\Omega),f\neq0\right\},
\end{eqnarray*}
and it is not difficult to get that
$\lambda^{\ast}_{1,\varphi}(\Omega)=\lambda_{1,\varphi}(\Omega)$ if
$\Omega$ has compact closure and its boundary $\partial\Omega$ is
piecewise smooth.

Domain monotonicity of eigenvalues with vanishing Dirichlet data
also holds for the first Dirichlet eigenvalue of $\Delta_{\varphi}$
(see, e.g., \cite[Lemma 1.5]{dm2}). This implies that for a complete
noncompact Riemannian manifold $M$, one can
 define the following limit
\begin{eqnarray*}
\lambda_{1,\varphi}(M):=\lim\limits_{\ell\rightarrow\infty}\lambda_{1,\varphi}\left(B_{M}(q,\ell)\right)=\lim\limits_{\ell\rightarrow\infty}\lambda^{\ast}_{1,\varphi}\left(B_{M}(q,\ell)\right),
\end{eqnarray*}
which is independent of the choice of the point $q$ and can be seen
as a generalization of $\lambda_{1}(M)$. Clearly,
$\lambda_{1,\varphi}(M)\geq0$ and if $\varphi=const.$, then
$\lambda_{1,\varphi}(M)=\lambda_{1}(M)$. Based on Schoen-Yau's
suggestion mentioned before, it is natural to ask:   \vspace{2mm}
$\\$  \textbf{Question 1}. \emph{For a given complete noncompact
Riemannian manifold $M$, under what conditions,
$\lambda_{1,\varphi}(M)>0$?} \vspace{2mm}
 $\\$For an
$n$-dimensional ($n\geq2$) complete noncompact submanifold of a
hyperbolic space whose norm of the mean curvature vector $\|H\|$
satisfies $\|H\|\leq\alpha<n-1$, Du and Mao \cite[Theorem 1.7]{dm2}
proved that if $\|\varphi\|\leq C$ \footnote{It is easy to know that
the constant $C$ satisfies $C<n-1-\alpha$, which is the potential
assumption in \cite[Theorem 1.7]{dm2}, since in the proof of
\cite[Theorem 1.7]{dm2}, the positive number $\epsilon$ is chosen to
be $\epsilon=(n-1-\alpha-C)/2$.}, then
$\lambda_{1,\varphi}(M)\geq\frac{(n-1-\alpha-C)^2}{4}$, with
equality attained when $M$ is totally geodesic and $\varphi=const.$,
which generalized Cheung-Leung's and Mckean's conclusions mentioned
before.

Consider the BVP
\begin{eqnarray}  \label{1-3}
\left\{
\begin{array}{ll}
\Delta_{p}u+\lambda|u|^{p-2}u=0  ~\qquad & {\rm{in}}~\Omega,\\
  u=0   ~ \qquad & {\rm{on}}~\partial\Omega,
  \end{array}
\right.
\end{eqnarray}
where $\Omega\subset M$ is an open bounded connected domain in a
given Riemannian manifold $M$,
$\Delta_{p}u:={\rm{div}}(\|\nabla{u}\|^{p-2}\nabla{u})$ is the
nonlinear $p$-Laplacian of $u$ with $1<p<\infty$. It is known that
(\ref{1-3}) has a positive weak solution, which is unique modulo the
scaling, in $W^{1,p}_{0}(\Omega)$, the completion of the set
$C^{\infty}_{0}(\Omega)$ of smooth functions compactly supported on
$\Omega$ under the Sobolev norm
$\|u\|_{1,p}=\{\int_{\Omega}(|u|^{p}+\|\nabla
u\|^{p})\}^{\frac{1}{p}}$, and the first Dirichlet eigenvalue
$\lambda_{1,p}(\Omega)$ of the $p$-Laplacian in the eigenvalue
problem (\ref{1-3}) can be characterized by
\begin{eqnarray*}
\lambda_{1,p}(\Omega)=\inf\left\{\frac{\int_{\Omega}\|\nabla
f\|^{p}}{\int_{\Omega}|f|^{p}}\Big{|}f\in
W^{1,p}_{0}(\Omega),f\neq0\right\}.
\end{eqnarray*}
The (closed or Dirichlet) eigenvalue problem of the $p$-Laplacian
has been studied by the first author here and some interesting
conclusions have been obtained (see, e.g., \cite{dm,dm2,m1,m2}).
Domain monotonicity of eigenvalues with vanishing Dirichlet data
also holds for the first Dirichlet eigenvalue of $\Delta_{p}$ (see,
e.g., \cite[Lemma 1.1]{dm2}). This implies that for a complete
noncompact Riemannian manifold $M$, one can
 define the following limit
\begin{eqnarray*}
\lambda_{1,p}(M):=\lim\limits_{\ell\rightarrow\infty}\lambda_{1,p}\left(B_{M}(q,\ell)\right),
\end{eqnarray*}
which is independent of the choice of the point $q$ and can be seen
as a generalization of $\lambda_{1}(M)$. Clearly,
$\lambda_{1,p}(M)\geq0$ and if $p=2$, then
$\lambda_{1,p}(M)=\lambda_{1}(M)$. Based on Schoen-Yau's suggestion
mentioned before, it is natural to ask :   \vspace{2mm} $\\$
\textbf{Question 2}. \emph{For a given complete noncompact
Riemannian manifold $M$, under what conditions,
$\lambda_{1,p}(M)>0$?} \vspace{2mm}
 $\\$ For an $n$-dimensional ($n\geq2$) complete
noncompact submanifold of a hyperbolic space whose norm of the mean
curvature vector $\|H\|$ satisfies $\|H\|\leq\alpha<n-1$, Du and Mao
\cite[Theorem 1.3]{dm2} proved
$\lambda_{1,p}(M)\geq\left(\frac{n-1-\alpha}{p}\right)^{p}>0$, with
equality attained when $M$ is totally geodesic and $p=2$, which
generalized Cheung-Leung's and Mckean's conclusions mentioned
before.

The purpose of this paper is trying to positively answer Questions 1
and 2 \emph{further}. In fact, we have obtained the following facts:
\begin{itemize}

\item By introducing a quantity $c(\Omega)$ for a domain $\Omega$ with
compact closure (see Definition \ref{dcq}), Bessa-Montenegro type
lower bounds for the weighted fundamental tone
$\lambda_{1,\varphi}^{\ast}(\Omega)$ and the first eigenvalue
$\lambda_{1,p}(\Omega)$ of the $p$-Laplacian can be obtained - see
Lemma \ref{lemma2-3}. By applying Hessian Comparison Theorem, Domain
monotonicity of eigenvalues with vanishing Dirichlet data for
$\lambda_{1,\varphi}^{\ast}(\cdot)$ and $\lambda_{1,p}(\cdot)$,
Bessa-Montenegro type lower bounds would give us Mckean-type lower
bounds for Hadamard manifolds with strictly negative sectional
curvature - see Lemma \ref{lemma2-8}.

\item Let $\phi:M\rightarrow Q$  be an isometric immersion from $n$-dimensional ($n\geq2$) Riemannian manifold to an $m$-dimensional Riemannian manifold,
 and moreover, $M$ has locally bounded mean curvature (see Definition
 \ref{lbmc}). For any connected component $\Omega$ of
$\phi^{-1}(\overline{B_{Q}(q,r)})$ with $q\in Q\backslash \phi(M)$,
and $r>0$, under different assumptions on sectional curvatures, some
strictly positive lower bounds have been obtained for the
fundamental tone $\lambda_{1,\varphi}^{\ast}(\Omega)$ (no matter
$\Omega$ is bounded or unbounded) and the first eigenvalue
$\lambda_{1,p}(\Omega)$ of the $p$-Laplacian (in this case, $\Omega$
is bounded and has piecewise smooth boundary) - see Theorem
\ref{theorem3-2}. As a direct consequence, if furthermore $M$ is
noncompact with bounded mean curvature (stronger than the
\emph{locally} bounded mean curvature assumption) and the sectional
curvature of $Q$ is bounded from above by some strictly negative
constant, then $\lambda_{1,\varphi}(M)$ and $\lambda_{1,p}(M)$ have
strictly positive lower bounds - see Corollary \ref{corollary3-4}.

\item Recently, because of the discovery of many interesting examples of minimal surfaces in product spaces $N\times\mathbb{R}$ (see, e.g., \cite{mr1,mr2}), the study of this kind of spaces has attracted
geometers' attention. Based on this, we investigate submanifolds
$\Omega$, with locally bounded mean curvature, of
$N\times\mathbb{R}$ and would like to know ``\emph{under what
conditions, $\lambda_{1,\varphi}^{\ast}(\Omega)>0$ and
$\lambda_{1,p}(\Omega)>0$?}". A positive answer has been given - see
Theorem \ref{theorem4-4} for details.

\item For an $n$-dimensional ($n\geq2$) complete non-compact $\varphi$-minimal submanifold $M$ of the
weighted manifold $\left(\mathbb{H}^{m}(-1),e^{-\varphi}dv\right)$,
where $\mathbb{H}^{m}(-1)$ is the hyperbolic $m$-space with
sectional curvature $-1$, $\varphi$ is a real-valued smooth function
on $\mathbb{H}^{m}(-1)$ and $dv$ is the volume element, a strictly
positive lower bound has been obtained for the first eigenvalue
$\lambda_{1,p}(M)$ of the $p$-Laplacian on $M$ - see Theorem
\ref{theorem-5-main} for details.

\item Interesting \emph{new} lower
bounds for the first Dirichlet eigenvalues of the weighted Laplacian
and the $p$-Laplacian on geodesic balls of complete Riemannian
manifolds have been given - see Theorem \ref{theorem6-1} for
details.

\end{itemize}

\section{Bessa-Montenegro type and Mckean-type lower bounds for the weighted fundamental tone and the first eigenvalue of the $p$-Laplacian}
\renewcommand{\thesection}{\arabic{section}}
\renewcommand{\theequation}{\thesection.\arabic{equation}}
\setcounter{equation}{0} \setcounter{maintheorem}{0}

By using a notion introduced in \cite{bm}, we can give lower bounds
for the weighted fundamental tone for arbitrary bounded domains, and
the lowest eigenvalue  for the Dirichlet eigenvalue problem of the
weighted Laplacian and the $p$-Laplacian on normal domains.

\begin{definition}  \label{dcq} (\cite{bm})
Let $\Omega\subset M$ be a domain with compact closure in a
$C^{\infty}$ Riemannian manifold $M$. Let $\mathcal{X}(\Omega)$ be
the set of all smooth vector fields $X$ on $\Omega$ with
$\|X\|_{\infty}:=\sup_{\Omega}\|X\|<\infty$ and $\inf\rm{div}X>0$
with $\rm{div}$ the divergence operator on $M$. Define $c(\Omega)$
by
 \begin{eqnarray} \label{OD}
 c(\Omega):=\sup\left\{\frac{\inf\rm{div}X}{\|X\|_{\infty}}:X\in\mathcal{X}(\Omega)\right\}.
 \end{eqnarray}
\end{definition}

\begin{remark}
\rm{ As shown in \cite[Remark 2.2]{bm}, it is easy to get that
$\mathcal{X}(\Omega)$ is not empty. This is because the boundary
value problem (BVP for short)
\begin{eqnarray*}
\left\{
\begin{array}{ll}
\Delta u=1,~~~~\rm{in}~~\Omega\\[2mm]
u=0,~~\quad\rm{on}~~\partial\Omega &\quad
\end{array}
\right.
\end{eqnarray*}
always has a solution on a bounded domain $\Omega\subset M$, and
then at least one can choose $X=\nabla{u}$, the gradient of $u$,
which implies that ${\rm{div}}(X)=1$ and $\|X\|<\infty$.
 }
\end{remark}

Now, we can prove the following.

\begin{lemma} \label{lemma2-3}
Let $\Omega\subset M$ be a domain with compact closure and nonempty
boundary (i.e., $\partial\Omega\neq\emptyset$) in a Riemannian
manifold $M$. Then we have
 \begin{eqnarray*}
\lambda^{\ast}_{1,\varphi}(\Omega)\geq\frac{\left(c(\Omega)-c^{+}\right)^{2}}{4}>0
 \end{eqnarray*}
 provided $\|\nabla\varphi\|\leq c^{+}< c(\Omega)$, where $c^{+}$ is the
 supremum of the norm of the gradient of $\varphi$ and is strictly less than $c(\Omega)$, and $c(\Omega)$ is given by
 (\ref{OD}). Moreover, if furthermore the boundary $\partial\Omega$
 is piecewise smooth, then we have
 \begin{eqnarray*}
\lambda_{1,p}(\Omega)\geq\left(\frac{c(\Omega)}{p}\right)^{p}>0.
\end{eqnarray*}
\end{lemma}

\begin{proof}
Taking $f\in\ C^{\infty}_{0}(\Omega)$, the set of all smooth
functions compactly supported on $\Omega$, and
$X\in\mathcal{X}(\Omega)$. By a direct calculation, we have
\begin{eqnarray} \label{2-1}
{\rm{div}}(|f|^{p}X)&=&\langle\nabla|f|^{p},X\rangle+|f|^{p}{\rm{div}}X\nonumber\\
&\geq&\ \ -p|f|^{p-1}\|\nabla
f\|{\rm{sup}}\|X\|+\inf{\rm{div}}X\cdot|f|^{p}.
\end{eqnarray}
By Young's inequality, one can obtain
\begin{eqnarray*}
 |f|^{p-1}\|\nabla
f\|=\epsilon|f|^{p-1}\cdot\frac{\|\nabla
f\|}{\epsilon}\leq\frac{\left(\frac{\|\nabla
f\|}{\epsilon}\right)^{p}}{p}+\frac{\left(\epsilon|f|^{p-1}\right)^{\frac{p}{p-1}}}{\frac{p}{p-1}},
\end{eqnarray*}
where $\epsilon>0$ is a parameter determined later. Substituting the
above inequality into (\ref{2-1}) yields
 \begin{eqnarray} \label{2-2}
{\rm{div}}(|f|^{p}X)\geq-p\sup\|X\|\left[\frac{\left(\frac{\|\nabla
f\|}{\epsilon}\right)^{p}}{p}+\frac{(\epsilon|f|^{p-1})^{\frac{p}{p-1}}}{\frac{p}{p-1}}\right]+\inf{\rm{div}}X\cdot|f|^{p}.
\end{eqnarray}
Choosing
 \begin{eqnarray*}
 \epsilon=\left(\frac{\inf{\rm{div}}X}{p\sup\|X\|}\right)^{\frac{p-1}{p}},
 \end{eqnarray*}
 in (\ref{2-2}), integrating both sides of (\ref{2-2}) over $\Omega$
 and using the divergence theorem, we have
 \begin{eqnarray} \label{add-1}
\int_{\Omega}\|\nabla
f\|^{p}\geq\left(\frac{\inf{\rm{div}}X}{p\sup\|X\|}\right)^{p}\int_{\Omega}|f|^{p},
\end{eqnarray}
which implies
\begin{eqnarray*}
\lambda_{1,p}(\Omega)\geq\left(\frac{c(\Omega)}{p}\right)^{p}
\end{eqnarray*}
by taking the supremum over all vector fields
$X\in\mathcal{X}(\Omega)$ to the RHS of (\ref{add-1}).

 If $\|\nabla \varphi\|\leq c^{+}<c(\Omega)$ with $c^{+}\geq0$ the
supremum of $\|\nabla \varphi\|$, then we have
\begin{eqnarray} \label{2-3}
{\rm{div}}(f^{2}Xe^{-\varphi})&=&e^{-\varphi}\langle\nabla f^{2},X\rangle+f^{2}e^{-\varphi}{\rm{div}}X-f^{2}e^{-\varphi} \langle \nabla \varphi,X \rangle\nonumber\\
&\geq&\ \ e^{-\varphi}\left[-2|f|\cdot\|\nabla
f\|\cdot\sup\|X\|+f^{2}\inf{\rm{div}}X-f^{2}c^{+}\sup\|X\|\right]\nonumber\\
&\geq& \ e^{-\varphi}\left[\left(-\epsilon f^{2}-\frac{\|\nabla
f\|^{2}}{\epsilon}\right)\sup\|X\|+f^{2}\inf{\rm{div}}X-f^{2}c^{+}\sup\|X\|\right],
\end{eqnarray}
where $\epsilon>0$ is a parameter determined later. Integrating both
sides of (\ref{2-3}) and using the divergence theorem, we have
\begin{eqnarray} \label{2-4}
\int_{\Omega}\|\nabla
f\|^{2}e^{-\varphi}\geq\frac{\epsilon(\inf{\rm{div}}X-c^{+}\sup\|X\|-\epsilon\sup\|X\|)}{\sup\|X\|}\int_{\Omega}f^{2}e^{-\varphi}.
\end{eqnarray}
On the other hand, since
\begin{eqnarray*}
\frac{\epsilon(\inf{\rm{div}}X-c^{+}\sup\|X\|-\epsilon\sup\|X\|)}{\sup\|X\|}\leq\left(\frac{\frac{\inf{\rm{div}}X}{\sup\|X\|}-c^{+}}{2}\right)^{2}
\end{eqnarray*}
with equality holds if and only if
$\epsilon=\frac{\inf{\rm{div}}X}{2\sup\|X\|}-\frac{c^{+}}{2}>0$, we
can obtain
\begin{eqnarray*}
 \lambda_{1,\varphi}(\Omega)\geq \frac{(c(\Omega)-c^{+})^{2}}{4}>0
\end{eqnarray*}
by choosing
$\epsilon=\frac{\inf{\rm{div}}X}{2\sup\|X\|}-\frac{c^{+}}{2}$ in
(\ref{2-4}) and by taking the supremum over all vector fields
$X\in\mathcal{X}(\Omega)$. This completes the proof of Lemma
\ref{lemma2-3}.
\end{proof}

\begin{remark} \label{remark2-4}
\rm{ (1) Clearly, when $p=2$ (or $\varphi=const.$), the nonlinear
$p$-Laplacian (or the weighted Laplacian) degenerates into the
Laplacian. Correspondingly,
$\lambda_{1,p}(\Omega)=\lambda_{1}^{\ast}(\Omega)$ (or
$\lambda_{1,\varphi}(\Omega)=\lambda_{1}^{\ast}(\Omega)$,
$c^{+}=0$), and moreover,
$\lambda_{1}^{\ast}(\Omega)\geq\left(\frac{c(\Omega)}{2}\right)^2$,
which is the lower bound for $\lambda_{1}^{\ast}(\Omega)$ in
\cite[Lemma 2.3]{bm} given by Bessa and Montenegro. Based on this
fact, we would like to use  \emph{Bessa-Montenegro type lower
bounds} to call the lower bounds for the lowest Dirichlet eigenvalue
(resp., the weighted fundamental tone) shown in Lemma
\ref{lemma2-3}. Besides, to prove Bessa-Montenegro type lower bounds
here, we only need to consider vector fields smooth almost every in
$\Omega$ such that $\int_{\Omega}{\rm{div}}(|f|^{p}X)=0$ or
$\int_{\Omega}{\rm{div}}(f^{2}Xe^{-\varphi})=0$ for all $f\in
C^{\infty}_{0}(\Omega)$. \\
(2) It has been shown in \cite[Remark 2.7]{bm} that $c(\Omega)\leq
h(\Omega)$ with
$h(\Omega):=\inf_{A\subset\Omega}\frac{\rm{vol}(\partial
A)}{\rm{vol}(A)}$ the Cheeger's constant. However, in some cases,
for instance, for balls in the Euclidean space or Hadamard
manifolds, $c(\Omega)=h(\Omega)$. The advantage of defining
$c(\Omega)$ is the computability of lower bounds for
$\lambda_{1,p}(\Omega)$, $\lambda_{1,\varphi}(\Omega)$ via any lower
bound for $c(\Omega)$, and this way can be applied to arbitrary
domains. Besides, we can use Lemma \ref{lemma2-3} to derive
\emph{Mckean-type} lower bounds below - see Lemma \ref{lemma2-8} for
details. }
\end{remark}

Applying Lemma \ref{lemma2-3}, one can get the following conclusion
directly.

\begin{corollary}
Let $\Omega\subset M$ be a normal domain with compact closure in a
smooth Riemannian manifold $M$. For the BVP
\begin{eqnarray*}
\left\{
\begin{array}{ll}
\Delta v=1,~~~~\rm{in}~~\Omega,\\[2mm]
v=0,~~~~\quad\rm{on}~~~\partial\Omega, &\quad
\end{array}
\right.
\end{eqnarray*}
we have
\begin{eqnarray*}
\lambda_{1,p}(\Omega)\geq\left(\frac{1}{p\|\nabla{v}\|_{\infty}}\right)^{p}>0.
\end{eqnarray*}
Besides,
 \begin{eqnarray*}
\lambda^{\ast}_{1,\varphi}(\Omega)\geq\frac{\left(\frac{1}{\|\nabla{v}\|_{\infty}}-c^{+}\right)^{2}}{4}>0
 \end{eqnarray*}
 provided $\|\nabla \varphi\|\leq c^{+}< \frac{1}{\|\nabla{v}\|_{\infty}}$,  where $c^{+}$ is the
 supremum of the norm of the gradient of $\varphi$ and is strictly less than
 $\frac{1}{\|\nabla{v}\|_{\infty}}$.
\end{corollary}

\begin{corollary} \label{IMPORTANT}
There are no smooth bounded vector fields $X:M\rightarrow TM$ with
$\inf_{M}{\rm{div}}X>0$ on complete noncompact manifolds $M$ such
that $\lambda_{1,p}(M)=0$ and $\lambda_{1,\varphi}(M)=0$. In
particular, there is no such vector field on $\mathbb{R}^n$.
\end{corollary}

As an interesting application of Lemma \ref{lemma2-3}, we can obtain
Mckean-type lower bounds for the first eigenvalue of the drifting
Laplacian and the $p$-Laplacian on the prescribed Hadamard manifold.
However, in order to prove that, we need to use \emph{Hessian
Comparison Theorem} below.

\begin{theorem} \label{HCT}
(Hessian Comparison Theorem) Let $M$ be a complete Riemannian
manifold and $x_{0},x\in M$.  Let $\gamma:[0,\rho(x)]\rightarrow M$
be a minimizing geodesic joining $x_{0}$ and $x$, where $\rho(x)$ is
the distance function ${\rm{dist}}_{M}(x_{0},x)$. Let $K$ be the
sectional curvatures of $M$ and $\mu_{i}(\rho)$, $i=0,1$, be
functions defined as below
\begin{eqnarray*}
\mu_{0}(\rho)=\left\{
\begin{array}{lll}
k_{0}\coth(k_{0}\rho(x)),~~~~\mathrm{if}~~\inf_{\gamma}K=-k_{0}^2,\\[2mm]
\frac{1}{\rho(x)},~~~~~~\qquad\qquad\mathrm{if}~~\inf_{\gamma}K=0,\\[2mm]
k_{0}\cot(k_{0}\rho(x)),~~~~~~\mathrm{if}~~\inf_{\gamma}K=k_{0}^{2}~\mathrm{and}~\rho<\frac{\pi}{2k_{0}}
&\quad
\end{array}
\right.
\end{eqnarray*}
and
\begin{eqnarray*}
\mu_{1}(\rho)=\left\{
\begin{array}{lll}
k_{1}\coth(k_{1}\rho(x)),~~~~\mathrm{if}~~\sup_{\gamma}K=-k_{1}^2,\\[2mm]
\frac{1}{\rho(x)},~~~~~~\qquad\qquad\mathrm{if}~~\sup_{\gamma}K=0,\\[2mm]
k_{1}\cot(k_{1}\rho(x)),~~~~~~\mathrm{if}~~\sup_{\gamma}K=k_{1}^{2}~\mathrm{and}~\rho<\frac{\pi}{2k_{1}}.
&\quad
\end{array}
\right.
\end{eqnarray*}
 Then the Hessian of $\rho$ and $\rho^{2}$ satisfy
\begin{eqnarray*}
&&\mu_{1}(\rho(x))\cdot\|X\|^{2}\leq\mathrm{Hess}\rho(x)(X,X)\leq\mu_{0}(\rho(x))\cdot\|X\|^{2},\\
&&\mathrm{Hess}\rho(x)(\gamma',\gamma')=0,\\
&&2\rho(x)\cdot\mu_{1}(\rho(x))\cdot\|X\|^{2}\leq\mathrm{Hess}\rho^{2}(x)(X,X)\leq2\rho(x)\cdot\mu_{0}(\rho(x))\cdot\|X\|^{2},\\
&&\mathrm{Hess}\rho^{2}(x)(\gamma',\gamma')=2,
\end{eqnarray*}
where $X$ is any vector in $T_{x}M$ perpendicular to
$\gamma'(\rho(x))$.
\end{theorem}
Hence, by applying Theorem \ref{HCT}, for the distance function
$\rho(x)$ on an $n$-dimensional Riemannian manifold $M$, we can get
\begin{eqnarray} \label{HCTLI}
2(n-1)\rho(x)\mu_{1}(\rho(x))+2\leq\Delta\rho^{2}(x)\leq2(n-1)\rho(x)\mu_{0}(\rho(x))+2.
\end{eqnarray}

\begin{lemma} \label{lemma2-8}
Let $M$ be an $n$-dimensional ($n\geq2$) Hadamard manifold whose
sectional curvature satisfies $K_{M}\leq-a^{2}<0$, $a>0$. Then we
have
\begin{eqnarray*}
\lambda_{1,p}(M)\geq\left[\frac{(n-1)\cdot a}{p}\right]^{p}>0.
\end{eqnarray*}
Moreover,
\begin{eqnarray*}
\lambda_{1,\varphi}(M)\geq\left[\frac{(n-1)\cdot
a-c^{+}}{2}\right]^{2}>0
\end{eqnarray*}
provided $\|\mathrm{grad}\varphi\|\leq c^{+}< (n-1)a$, where $c^{+}$
is the supremum of the norm of the gradient of $\varphi$ and is
strictly less than $(n-1)a$.
\end{lemma}

\begin{proof}
Let $\rho: M\rightarrow\mathbb{R}$ be the distance function to a
point $p\in M\setminus\Omega$ with $\Omega$ a normal domain in $M$,
and let $X=\rm{grad}\rho$. By (\ref{HCTLI}), we have
\begin{eqnarray*}
\Delta\rho(x)={\rm{div}}X\geq(n-1)\cdot
a\cdot\coth(a\cdot\rho(x))\geq(n-1)\cdot a.
\end{eqnarray*}
By Lemma \ref{lemma2-3}, it follows that
\begin{eqnarray*}
\lambda_{1,p}(\Omega)\geq\left[\frac{(n-1)\cdot a}{p}\right]^{p}
\end{eqnarray*}
and
\begin{eqnarray*}
\lambda_{1,\varphi}(\Omega)=\lambda^{\ast}_{1,\varphi}(\Omega)\geq\left[\frac{(n-1)\cdot
a-c^{+}}{2}\right]^{2},
\end{eqnarray*}
which, by \cite[Lemma 1.1]{dm2}, implies the lower bounds for
$\lambda_{1,p}(M)$, $\lambda_{1,\varphi}(M)$ in Lemma
\ref{lemma2-8}.
\end{proof}

\begin{remark}
\rm{ Clearly, when $p=2$ (or $\varphi=const.$), the nonlinear
$p$-Laplacian (or the weighted Laplacian) degenerates into the
Laplacian. Correspondingly, $\lambda_{1,p}(M)=\lambda_{1}(M)$ (or
$\lambda_{1,\varphi}(M)=\lambda_{1}(M)$, $c^{+}=0$), and moreover,
$\lambda_{1}(M)\geq\frac{(n-1)^{2}a^{2}}{4}>0$, which is exactly
Mckean's lower bound shown in \cite{hpm}. }
\end{remark}

\section{Eigenvalue estimates for submanifolds with locally bounded mean curvature in Hadamard manifolds}
\renewcommand{\thesection}{\arabic{section}}
\renewcommand{\theequation}{\thesection.\arabic{equation}}
\setcounter{equation}{0} \setcounter{maintheorem}{0}

Let $\phi:M\rightarrow Q$ be an isometric immersion with $M$, $Q$
complete Riemannian manifolds, $\mathrm{dim}(M)=n$, $n\geq2$.
Consider a smooth function $g:Q\rightarrow \mathbb{R}$ and the
composition $f=g\circ \phi:M\rightarrow \mathbb{R}$. As before, let
$\Delta$ be the
 Laplace operator on $M$. \emph{However, because of the isometric immersion, for convenience, in this section, we can use ${\rm{grad}}(\cdot)$ to denote
 the gradient of a given function on $M$ or its isometric image $\phi(M)\subseteq Q$}. Identify $X$ with $d\phi(X)$, and
then we can obtain that at $q\in M$,
\begin{eqnarray*}
\langle{\rm{grad}}f,X\rangle=df(X)=dg(X)=\langle{\rm{grad}}g,X\rangle
\end{eqnarray*}
for every $X\in T_{q}M$. Therefore, it follows that
\begin{eqnarray*}
{\rm{grad}}g={\rm{grad}}f+({\rm{grad}}g)^{\perp},
\end{eqnarray*}
with $({\rm{grad}}g)^{\perp}$ perpendicular to $T_{q}M$. For $X,Y\in
T_{q}M$, let $\alpha(q)(X,Y)$ and ${\rm{Hess}}f(q)(X,Y)$ be the
second fundamental form of the immersion $\phi$ and the Hessian of
$f$ at $q\in M$, respectively. By the Gauss equation, we have
\begin{eqnarray} \label{3-1}
{\rm{Hess}}f(q)(X,Y)={\rm{Hess}}g(\phi(q))(X,Y)+\langle{\rm{grad}}g,\alpha(X,Y)\rangle_{\phi(q)}.
\end{eqnarray}
Taking the trace in (\ref{3-1}) w.r.t. an orthonormal basis
$\{e_{1},e_{2}\cdots e_{n}\}$ of $T_{q}M$, we can get
\begin{eqnarray} \label{3-2}
\Delta
f(q)=\sum\limits_{i=1}^{n}{\rm{Hess}}f(q)(e_{i},e_{i})=\sum\limits_{i=1}^{n}{\rm{Hess}}g(\phi(q))(e_{i},e_{i})+\left\langle
 {\rm{grad}}g,\sum\limits_{i=1}^{n}\alpha(e_{i},e_{i})\right\rangle.
\end{eqnarray}
See, e.g., \cite{cl,dm2} for more generalized versions of the
formulas (\ref{3-1}) and (\ref{3-2}) above.

We need the following notion.

\begin{definition} \label{lbmc}
An isometric immersion $\phi:M\rightarrow Q$ has locally bounded
mean curvature $H$ if for any $q\in Q$ and $r>0,$ the number
$h(q,r):=\sup\{\|H(x)\|; x\in\phi(M)\cap B_{Q}(q,r)\}$ is finite,
where, similar as before, $B_{Q}(q,r)$ denotes the geodesic ball,
with center $q$ and radius $r$, on $Q$.
\end{definition}

By using Lemma \ref{lemma2-3}, Theorem \ref{HCT} and the locally
bounded mean curvature assumption, we can prove the following.

\begin{theorem} \label{theorem3-2}
Let $\phi:M\rightarrow Q$  be an isometric immersion with locally
bounded mean curvature and let $\Omega$ be any connected component
of $\phi^{-1}(\overline{B_{Q}(q,r)})$, where $q\in Q\backslash
\phi(M)$,  $r>0$ and $\mathrm{dim}(M)=n$, $n\geq2$. Let
$\kappa(q,r)=\sup\{K_{Q}(x)|x\in B_{Q}(q,r)\},$ where $K_{Q}(x)$ is
the sectional curvature at $x$. Denote by $\mathrm{inj}(q)$ the
injectivity radius of $Q$ at the point $q$. Assume that $\varphi$ is
a real-valued smooth function on $M$ with
$\|\mathrm{grad}\varphi\|\leq c^{+}$, where $c^{+}$ is the supremum
of the norm of the gradient of $\varphi$. Choosing $r$ properly, we
have the following estimates:
  $\\$(1) If
$\kappa(q,\mathrm{inj}(q))=k^{2}<\infty$, $k>0$, choose
\begin{eqnarray*}
r<\min\left\{\mathrm{inj}(q),\frac{\pi}{2k},\cot^{-1}\left[\frac{h(q,\mathrm{inj}(q))}{(n-1)k}\right]\Big{/}k\right\}.
\end{eqnarray*}
Then we have
\begin{eqnarray*}
\lambda_{1,\varphi}^{\ast}(\Omega)\geq\left[\frac{(n-1)k\cot(kr)-h(q,r)-c^{+}}{2}\right]^{2}
\end{eqnarray*}
provided $c^{+}<(n-1)k\cot(kr)-h(q,r)$. If furthermore the boundary
$\partial\Omega$
 is piecewise smooth, then we have
 \begin{eqnarray*}
\lambda_{1,p}(\Omega)\geq\left[\frac{(n-1)k\cot(kr)-h(q,r)}{p}\right]^{p}.
\end{eqnarray*}
 $\\$(2) If $\lim\limits_{\ell\rightarrow\infty}\kappa(q,\ell)=\infty$, let
 \begin{eqnarray*}
 r(s):=\min\left\{\frac{\pi}{2\sqrt{\kappa(q,s)}},\cot^{-1}\left[\frac{h(q,s)}{(n-1)\sqrt{\kappa(q,s)}}\right]\Big{/}\sqrt{\kappa(q,s)}\right\},\qquad
 s>0.
 \end{eqnarray*}
Choose $r=\max\limits_{s>0}r(s)$. We have
\begin{eqnarray*}
\lambda_{1,\varphi}^{\ast}(\Omega)\geq\left[\frac{(n-1)\sqrt{\kappa(q,s)}\cot(\sqrt{\kappa(q,s)}r)-h(q,r)-c^{+}}{2}\right]^{2}
\end{eqnarray*}
provided
$c^{+}<(n-1)\sqrt{\kappa(q,s)}\cot(\sqrt{\kappa(q,s)}r)-h(q,r)$. If
furthermore the boundary $\partial\Omega$
 is piecewise smooth, then we have
 \begin{eqnarray*}
\lambda_{1,p}(\Omega)\geq\left[\frac{(n-1)\sqrt{\kappa(q,s)}\cot(\sqrt{\kappa(q,s)}r)-h(q,r)}{p}\right]^{p}.
\end{eqnarray*}
$\\$(3) If $\kappa(q,\mathrm{inj}(q))=0$, choose
$r<\min\left\{\mathrm{inj}(q),\frac{n}{h(q,\mathrm{inj}(q))}\right\}$.
Assume that $\frac{n}{h(q,\mathrm{inj}(q))}=\infty$ if
$h(q,\mathrm{inj}(q))=0$. Then we have
\begin{eqnarray*}
\lambda_{1,\varphi}^{\ast}(\Omega)\geq\left[\frac{\frac{n}{r}-h(q,r)-c^{+}}{2}\right]^{2}
\end{eqnarray*}
provided $c^{+}<\frac{n}{r}-h(q,r)$. If furthermore $\Omega$ is
bounded and its boundary $\partial\Omega$
 is piecewise smooth, then we have
\begin{eqnarray*}
\lambda_{1,p}(\Omega)\geq\left[\frac{\frac{n}{r}-h(q,r)}{p}\right]^{p}.
\end{eqnarray*}
$\\$(4) If $\kappa(q,\mathrm{inj}(q))=-k^{2}<\infty$, $k>0$, and
$h(q,\mathrm{inj}(q))<(n-1)k$, choose $r<\mathrm{inj}(q)$. Then
\begin{eqnarray*}
\lambda_{1,\varphi}^{\ast}(\Omega)\geq\left[\frac{(n-1)k-h(q,r)-c^{+}}{2}\right]^{2}
\end{eqnarray*}
provided $c^{+}<(n-1)k-h(q,r)$. If furthermore $\Omega$ is bounded
and its boundary $\partial\Omega$
 is piecewise smooth, then we have
 \begin{eqnarray*}
\lambda_{1,p}(\Omega)\geq\left[\frac{(n-1)k-h(q,r)}{p}\right]^{p}.
 \end{eqnarray*}
 $\\$(5) If If $\kappa(q,\mathrm{inj}(q))=-k^{2}<\infty$, $k>0$, and
$h(q,\mathrm{inj}(q))\geq(n-1)k$, choose
\begin{eqnarray*}
r<\min\left\{\mathrm{inj}(q),\coth^{-1}\left[\frac{h(q,\mathrm{inj}(q))}{(n-1)k}\right]\Big{/}k\right\}.
\end{eqnarray*}
Then we have
\begin{eqnarray*}
\lambda_{1,\varphi}^{\ast}(\Omega)\geq\left[\frac{(n-1)k\coth(kr)-h(q,r)-c^{+}}{2}\right]^{2}
\end{eqnarray*}
provided $c^{+}<(n-1)k\coth(kr)-h(q,r)$. If furthermore the boundary
$\partial\Omega$
 is piecewise smooth, then we have
 \begin{eqnarray*}
\lambda_{1,p}(\Omega)\geq\left[\frac{(n-1)k\coth(kr)-h(q,r)}{p}\right]^{p}.
\end{eqnarray*}
In (2), since $r(s)>0$ for small $s$, $r>0$. In (3)-(5), because of
the non-positivity assumption on $\kappa(q,\mathrm{inj}(q))$, the
radius $r$ is not necessary to be finite, which implies that the
connected component $\Omega$ of $\phi^{-1}(\overline{B_{Q}(q,r)})$
may be unbounded as $r\rightarrow\infty$. Besides, in (4), one can
have a slight better estimate as follows
\begin{eqnarray*}
\lambda_{1,\varphi}^{\ast}(\Omega)\geq\left[\frac{(n-1)k+\frac{1}{r}-h(q,r)-c^{+}}{2}\right]^{2}
\end{eqnarray*}
provided $c^{+}<(n-1)k+\frac{1}{r}-h(q,r)$, by choosing
$X=\mathrm{grad}(\rho^{2}\circ\phi)$ in the proof below.
\end{theorem}

\begin{proof}
Similar to the proof of \cite[Theorem 4.3]{bm}. Define two functions
as follows
\begin{eqnarray*}
f_{i}=\rho^{i}\circ\phi:M\rightarrow\mathbb{R}, \qquad i=1,2,
\end{eqnarray*}
where $\rho(x)={\rm{dist}}_{Q}(q,x)$ is the distance function on
$Q$. Clearly, $f_{1}$, $f_{2}$ are smooth functions on
$\phi^{-1}\left(B_{Q}(q,\mathrm{inj}(q))\right)$. Let $\Omega$ be a
connected component of
$\phi^{-1}(\overline{B_{Q}(q,r)})\subseteq\phi^{-1}\left(B_{Q}(q,\mathrm{inj}(q))\right)$,
and let $X_{i}=\mathrm{grad}f_{i}$, $i=1,2$, on $\Omega$. By
(\ref{3-2}), we have
\begin{eqnarray*}
\mathrm{div}X_{i}(x)=\Delta
f_{i}(x)=\sum\limits_{j=1}^{n-1}\mathrm{Hess}\rho^{i}(\phi(x))(e_{j},e_{j})+\langle\mathrm{grad}\rho^{i},H\rangle_{\phi(x)},
\end{eqnarray*}
with $\{e_{1},e_{2},\ldots,e_{n}\}$ an orthonormal basis of
$T_{x}M$, where $e_{n}=\mathrm{grad}\rho(x)$. Applying Theorem
\ref{HCT} directly, one can obtain
\begin{itemize}
\item if $\kappa(q,\mathrm{inj}(q))=k^{2}<\infty$, $k>0$, then
$\mathrm{div}X_{1}\geq (n-1)k\cot(kr)-h(q,r)>0$;

\item if $\kappa(q,\mathrm{inj}(q))=0$, then
$\mathrm{div}X_{2}\geq2n-2rh(q,r)>0$;

\item if $\kappa(q,\mathrm{inj}(q))=-k^{2}<\infty$, $k>0$, then
$\mathrm{div}X_{1}\geq (n-1)k\coth(kr)-h(q,r)>0$.
\end{itemize}
Together with the fact that $\|X_{1}\|=1$ and $\|X_{2}\|=2r$,
estimates in Theorem \ref{theorem3-2} can be obtained by applying
Lemma \ref{lemma2-3} directly.
\end{proof}

\begin{remark}
\rm{Clearly, when $\varphi=const.$ (or $p=2$, $\Omega$ is bounded),
our estimates here are exactly those in \cite[Theorem 4.3]{bm}.}
\end{remark}

 Applying directly Theorem \ref{theorem3-2}, we can obtain

 \begin{corollary}
 Let $\phi:M\rightarrow\mathbb{R}^{m}$  be an isometric
 minimal immersion of an $n$-dimensional ($n\geq2$) complete
 submanifold. Assume that $\phi(M)\subset B_{\mathbb{R}^{m}}(o,r)$, then
 $\lambda_{1,p}(M)\geq\left(\frac{n}{pr}\right)^{p}$.
 \end{corollary}

Using a similar proof to that of \cite[Corollary 4.4]{bm} and
applying directly Theorem \ref{theorem3-2}, \cite[Proposition
10.1]{ag}, \cite[Theorem A.3]{sy},
 we can get the following.

\begin{corollary} \label{corollary3-4}
Let $\phi:M\rightarrow Q$  be an isometric immersion with bounded
mean curvature $\|H\|\leq\alpha<(n-1)a$, where $M$ is an
$n$-dimensional complete noncompact Riemannian manifold and $Q$ is
an $m$-dimensional complete simply connected Riemannian manifold
with sectional curvature $K_{Q}$ satisfying $K_{Q}\leq-a^{2}<0$ for
some constant $a>0$. Assume that $\varphi$ is a real-valued smooth
function on $M$ with $\|\mathrm{grad}\varphi\|\leq c^{+}$, where
$c^{+}$ is the supremum of the norm of the gradient of $\varphi$.
Then we have the following estimates
\begin{eqnarray*}
\lambda_{1,\varphi}(M)\geq\left[\frac{(n-1)a-\alpha-c^{+}}{2}\right]^{2}>0
\qquad (provided ~c^{+}<(n-1)a-\alpha)
\end{eqnarray*}
and
\begin{eqnarray*}
\lambda_{1,p}(M)\geq\left[\frac{(n-1)a-\alpha}{p}\right]^{p}>0.
\end{eqnarray*}
In particular, there exist entire Green's functions on $M$. If
furthermore $M$ is minimal, then $M$ is non-parabolic.
\end{corollary}

\begin{remark}
\rm{ Corollary \ref{corollary3-4} gives a positive answer to
Questions 1 and 2 proposed in Section \ref{section1}, i.e., finding
conditions such that $\lambda_{1,\varphi}(M)>0$,
$\lambda_{1,p}(M)>0$ for a complete noncompact manifold $M$, and
also shows interesting geometric conclusions, i.e., the existence of
Green's functions and the non-parabolic property. Besides, if
$Q=\mathbb{H}^{m}(-1)$ which implies $a=1$, then our lower bounds
here are exactly those in \cite[Theorems 1.3 and 1.7]{dm2}. }
\end{remark}

\section{Eigenvalue estimates for submanifolds with locally bounded mean curvature in product manifolds $N\times\mathbb{R}$}
\renewcommand{\thesection}{\arabic{section}}
\renewcommand{\theequation}{\thesection.\arabic{equation}}
\setcounter{equation}{0} \setcounter{maintheorem}{0}

Let $\phi:M\rightarrow N\times\mathbb{R}$ be an isometric immersion
from an $n$-dimensional complete Riemannian manifold to the product
space $N\times\mathbb{R}$ with $N$ an $m$-dimensional complete
Riemannian manifold. Since $\phi$ is an isometric immersion, we have
formulas $(\ref{3-1})$, $(\ref{3-2})$ with $Q=N\times\mathbb{R}$.
Besides, for convenience, \emph{we can use ${\rm{grad}}(\cdot)$ to
denote
 the gradient of a given function on $M$ or its isometric image $\phi(M)\subseteq
 N\times\mathbb{R}$.} In this section, we would like to estimate from below the first
fundamental tone $\lambda_{1,\varphi}^{\ast}(\Omega)$ of $\Omega$
(with $\Omega\subseteq  M$) and the first eigenvalue
$\lambda_{1,p}(\Omega)$ of the $p$-Laplacian on $\Omega$ (with
$\Omega\subset M$ a domain with compact closure and piecewise smooth
boundary). However, before that, we need the following notion, which
is stronger than the one in Definition \ref{lbmc}.

\begin{definition} \label{lbmc-2} (\cite{bc})
An isometric immersion $\phi:M\rightarrow N\times\mathbb{R}$ has
locally bounded mean curvature $H$ if for any $q\in N$ and $r>0,$
the number $h(q,r):=\sup\{\|H(x)\|;
x\in\phi(M)\cap\left(B_{N}(q,r)\times\mathbb{R}\right)\}$ is finite,
where, similar as before, $B_{N}(q,r)$ denotes the geodesic ball,
with center $q$ and radius $r$, on $N$.
\end{definition}

We also need the following conclusion, which is an extension of
\cite[Theorem 1.7]{bm2}.
\begin{lemma} \label{lemma4-2}
Let $\mathcal{W}^{1,1}(M)$ be the Sobolev space of all vector fields
$X\in L^{1}_{loc}(M)$ possessing weak divergence \footnote{For a
Riemannian manifold $M$, a function $g\in L^{1}_{loc}(M)$ is a weak
divergence of $X$ if
$\int_{M}g\psi=-\int_{M}\langle\mathrm{grad}\psi,X\rangle$, $\forall
\psi\in C^{\infty}_{0}(M)$. There exists at most one $g\in
L^{1}_{loc}(M)$ for a given vector field $X\in L^{1}_{loc}(M)$ and
we can write $g=\mathrm{div}X$. Clearly, for a $C^{1}$ vector field
$X$, its classical divergence coincides with the weak divergence
$\mathrm{div}X$. } $\mathrm{div}X$ on a Riemannian manifold $M$.
Assume that $\varphi$ is a real-valued smooth function on $M$ with
$\|\mathrm{grad}\varphi\|\leq c^{+}$, where $c^{+}$ is the supremum
of the norm of the gradient of $\varphi$. Then the weighted
fundamental tone $\lambda_{1,\varphi}^{\ast}(M)$ of $M$ satisfies
\begin{eqnarray} \label{add-o-1}
\lambda_{1,\varphi}^{\ast}(M)\geq\sup\limits_{\mathcal{W}^{1,1}(M)}\left\{\inf\limits_{M}\left(\mathrm{div}X-\|X\|^{2}-c^{+}\|X\|\right)\right\}.
\end{eqnarray}
If furthermore $M$ is complete, then the first eigenvalue
$\lambda_{1,p}(M)$ of the $p$-Laplacian satisfies
\begin{eqnarray} \label{add-o-2}
\lambda_{1,p}(M)\geq\sup\limits_{\mathcal{W}^{1,1}(M)}\left\{\inf\limits_{M}\left[\mathrm{div}X-(p-1)\|X\|^{\frac{p}{p-1}}\right]\right\}.
\end{eqnarray}
\end{lemma}

\begin{proof}
Let $X\in L^{1}_{loc}(M)$ and $f\in C^{\infty}_{0}(M)$. Clearly, we
have $\int_{M}\mathrm{div}(f^{2}Xe^{-\varphi})=0$ and
$\int_{M}\mathrm{div}(|f|^{p}X)=0$. By a direct computation, it
follows that
\begin{eqnarray*}
0=\int_{M}\mathrm{div}(f^{2}Xe^{-\varphi})&=&\int_{M}f^{2}\mathrm{div}X\cdot
e^{-\varphi}+\int_{M}\left\langle\mathrm{grad}f^{2},X\right\rangle
e^{-\varphi}-\int_{M}f^{2}\left\langle\mathrm{grad}\varphi,X\right\rangle
e^{-\varphi}\\
&\geq&\int_{M}f^{2}\mathrm{div}X\cdot
e^{-\varphi}-2\int_{M}|f|\cdot\|X\|\cdot\|\mathrm{grad}f\|e^{-\varphi}-c^{+}\int_{M}\|X\|f^{2}e^{-\varphi}\\
&\geq&\int_{M}f^{2}\mathrm{div}X\cdot
e^{-\varphi}-\int_{M}\left[f^{2}\cdot\|X\|^{2}+\|\mathrm{grad}f\|^{2}\right]e^{-\varphi}-c^{+}\int_{M}\|X\|f^{2}e^{-\varphi}\\
&=&\int_{M}\left(\mathrm{div}X-\|X\|^{2}-c^{+}\|X\|\right)f^{2}e^{-\varphi}-\int_{M}\|\mathrm{grad}f\|^{2}e^{-\varphi}\\
&\geq&\inf\limits_{M}\left(\mathrm{div}X-\|X\|^{2}-c^{+}\|X\|\right)\int_{M}f^{2}e^{-\varphi}-\int_{M}\|\mathrm{grad}f\|^{2}e^{-\varphi},
\end{eqnarray*}
which implies
\begin{eqnarray*}
\frac{\int_{M}\|\mathrm{grad}f\|^{2}e^{-\varphi}}{\int_{M}f^{2}e^{-\varphi}}\geq\inf\limits_{M}\left(\mathrm{div}X-\|X\|^{2}-c^{+}\|X\|\right)
\end{eqnarray*}
Then, by taking supremum to both sides of the above inequality over
$\mathcal{W}^{1,1}(M)$, we have
\begin{eqnarray*}
\frac{\int_{M}\|\mathrm{grad}f\|^{2}e^{-\varphi}}{\int_{M}f^{2}e^{-\varphi}}\geq\sup\limits_{\mathcal{W}^{1,1}(M)}\left\{\inf\limits_{M}\left(\mathrm{div}X-\|X\|^{2}-c^{+}\|X\|\right)\right\},
\end{eqnarray*}
which implies (\ref{add-o-1}). On the other hand, since
$\int_{M}\mathrm{div}(|f|^{p}X)=0$, by a direct calculation, one can
obtain
\begin{eqnarray*}
0=\int_{M}\mathrm{div}(|f|^{p}X)&=&\int_{M}\left\langle\mathrm{grad}\left(|f|^{p}\right),X\right\rangle+\int_{M}|f|^{p}\mathrm{div}X\\
&\geq&-\int_{M}p|f|^{p-1}\|\mathrm{grad}f\|\cdot\|X\|+\int_{M}|f|^{p}\mathrm{div}X\\
&\geq&-\int_{M}p\left[\frac{\left(|f|^{p-1}\|X\|\right)^{\frac{p}{p-1}}}{\frac{p}{p-1}}+\frac{\|\mathrm{grad}f\|^{p}}{p}\right]+\int_{M}|f|^{p}\mathrm{div}X\\
&=&\int_{M}\left[\mathrm{div}X-(p-1)\|X\|^{\frac{p}{p-1}}\right]|f|^{p}-\int_{M}\|\mathrm{grad}f\|^{p}\\
&\geq&\inf\limits_{M}\left[\mathrm{div}X-(p-1)\|X\|^{\frac{p}{p-1}}\right]\int_{M}|f|^{p}-\int_{M}\|\mathrm{grad}f\|^{p},
\end{eqnarray*}
where the second inequality holds by applying Young's inequality.
Therefore, we have
\begin{eqnarray*}
\frac{\int_{M}\|\mathrm{grad}f\|^{p}}{\int_{M}|f|^{p}}\geq\inf\limits_{M}\left[\mathrm{div}X-(p-1)\|X\|^{\frac{p}{p-1}}\right],
\end{eqnarray*}
and then, by taking supremum to both sides of the above inequality
over $\mathcal{W}^{1,1}(M)$, we have
\begin{eqnarray} \label{4-3}
\frac{\int_{M}\|\mathrm{grad}f\|^{p}}{\int_{M}|f|^{p}}\geq\sup\limits_{\mathcal{W}^{1,1}(M)}\left\{\inf\limits_{M}\left[\mathrm{div}X-(p-1)\|X\|^{\frac{p}{p-1}}\right]\right\}
\end{eqnarray}
which implies (\ref{add-o-2}). This completes the proof of Lemma
\ref{lemma4-2}.
\end{proof}

\begin{remark}
\rm{(1) Using an almost same method, we can get
\begin{eqnarray*}
\lambda_{1,\varphi}^{\ast}(M)\geq\sup\limits_{\mathcal{W}^{1,1}(M)}\left\{\inf\limits_{M\setminus{F}}\left(\mathrm{div}X-\|X\|^{2}-c^{+}\|X\|\right)\right\}
\end{eqnarray*}
and
\begin{eqnarray*}
\lambda_{1,p}(M)\geq\sup\limits_{\mathcal{W}^{1,1}(M)}\left\{\inf\limits_{M\setminus{F}}\left[\mathrm{div}X-(p-1)\|X\|^{\frac{p}{p-1}}\right]\right\},
\end{eqnarray*}
where $F$ has zero Riemannian volume.\\
(2) If $M$ is compact, then, by taking infimum to the LHS of
(\ref{4-3}) over the space $\{f|f\in W^{1,p}_{0}(\Omega), f\neq0\}$,
one can get (\ref{add-o-2}) directly. If $M$ is noncompact, one can
choose an exhaustion $\{\Omega_{i}\}_{i=1,2,3,\cdots}$ with
$\Omega_{i}\subset\Omega_{j}$, $i<j$, then as the compactness
situation, one can obtain
\begin{eqnarray*}
\lambda_{1,p}(\Omega_{i})\geq\sup\limits_{\mathcal{W}^{1,1}(\Omega_{i})}\left\{\inf\limits_{\Omega_{i}}\left[\mathrm{div}X-(p-1)\|X\|^{\frac{p}{p-1}}\right]\right\},
\end{eqnarray*}
which, by applying domain monotonicity of the first eigenvalue of
the $p$-Laplacian with vanishing Dirichlet data and taking limits to
both sides of the above inequality as $i\rightarrow\infty$, implies
(\ref{add-o-2}). }
\end{remark}

For clarifying argument below better, we need to define functions
$S_{k}(t)$ and $C_{k}(t)$ as follows.
\begin{eqnarray} \label{add-2}
S_{k}(t)=\left\{
\begin{array}{lll}
\sin(\sqrt{k}\cdot t)/\sqrt{k},~~~~~\qquad\mathrm{if}~~k>0,\\[2mm]
t,~~~~~~~\quad\qquad\qquad\qquad\mathrm{if}~~k=0,\\[2mm]
\sinh(\sqrt{-k}\cdot t)/\sqrt{-k},~~~~~\mathrm{if}~~k<0, &\quad
\end{array}
\right.
\end{eqnarray}
and
\begin{eqnarray} \label{add-3}
C_{k}(t)=S_{k}'(t).
\end{eqnarray}

We can prove the following.

\begin{theorem}  \label{theorem4-4}
Let $\phi:M\rightarrow N\times\mathbb{R}$ be an $n$-dimensional
($n\geq3$) complete minimal isometric immersed submanifold, where
the $m$-dimensional Riemannian manifold $N$ has radial sectional
curvature $K_{\gamma(t)}(\gamma'(t),\vec{v})\leq k$, $\vec{v}\in
T_{\gamma(t)}N$, $\|\vec{v}\|=1$,
$\vec{v}\bot\frac{\partial}{\partial{t}}$, along the minimizing
geodesic $\gamma(t)$ issuing from a point $q\in N$. Let $\Omega$ be
any connected component of
$\phi^{-1}\left(B_{N}(q,r)\times\mathbb{R}\right)\}$, where
$r<\min\left\{\mathrm{inj}_{N}(q),\frac{\pi}{2\sqrt{k}}\right\}$
($\pi/2\sqrt{k}=\infty$ if $k\leq0$), and $\mathrm{inj}_{N}(q)$
denotes the injectivity radius of $N$ at the point $q$. Assume that
$\varphi$ is a real-valued smooth function on $M$ with
$\|\mathrm{grad}\varphi\|\leq c^{+}$, where $c^{+}$ is the supremum
of the norm of the gradient of $\varphi$. Suppose in addition that
\begin{itemize}
\item if $|h(q,r)|<\digamma^{2}<\infty$, then
$r\leq(\frac{C_{k}}{S_{k}})^{-1}\cdot\frac{\digamma^{2}}{(n-2)}$ or

\item if $\lim\limits_{r\rightarrow\infty}h(q,r_{0})=\infty$, then
$r\leq(\frac{C_{k}}{S_{k}})^{-1}\cdot\frac{h(q,r_{0})}{(n-2)}$,
where $r_{0}$ is chosen such that
$(n-2)\frac{C_{k}(r_{0})}{S_{k}(r_{0})}-h(q,r_{0})=0$.
\end{itemize}
Then we have
\begin{eqnarray*}
\lambda_{1,\varphi}^{\ast}(\Omega)\geq\left[\frac{(n-2)\frac{C_{k}(r)}{S_{k}(r)}-h(q,r)-c^{+}}{2}\right]^{2}
\end{eqnarray*}
provided $c^{+}<(n-2)\frac{C_{k}(r)}{S_{k}(r)}-h(q,r)$, and
\begin{eqnarray*}
\lambda_{1,p}(\Omega)\geq\left[\frac{(n-2)\frac{C_{k}(r)}{S_{k}(r)}-h(q,r)}{p}\right]^{p}.
\end{eqnarray*}
\end{theorem}

\begin{proof}
Define a function
$\widetilde{\rho}:N\times\mathbb{R}\rightarrow\mathbb{R}$ by
$\widetilde{\rho}(x,t)=\rho_{N}(x)$, where
$\rho_{N}(x)=\mathrm{dist}_{N}(q,x)$ is the distance function in $N$
to the point $x_{0}$. Let
$\Omega\subset\phi^{-1}\left(B_{N}(q,r)\times\mathbb{R}\right)$,
$f=\widetilde{\rho}\circ\phi$ and $X=\mathrm{grad}f$. Properly
choose $r$ such that $\inf_{\Omega}\mathrm{div}X>0$. As before,
denote by $\Delta$ the Laplacian on $M$. Clearly, $\Delta
f=\mathrm{div}X$. By Lemma \ref{lemma2-3}, we have
\begin{eqnarray} \label{addd-1}
\lambda_{1,\varphi}^{\ast}(\Omega)\geq\left(\frac{\inf{\rm{div}}X}{2\sup\|X\|}-\frac{c^{+}}{2}\right)^{2}
\end{eqnarray}
and
\begin{eqnarray} \label{addd-2}
\lambda_{1,p}(\Omega)\geq\left(\frac{\inf{\rm{div}}X}{p\sup\|X\|}\right)^{p}.
\end{eqnarray}
Consider the orthonormal basis
$\left\{\mathrm{grad}\rho_{N},\frac{\partial}{\partial\theta_{1}},\cdots,\frac{\partial}{\partial\theta_{m-1}},\frac{\partial}{\partial{s}}\right\}$
for the tangent space $T_{(q,s)}(N\times\mathbb{R})$ with
$\phi(w)=(q,s)$, where
$\left\{\mathrm{grad}\rho_{N},\frac{\partial}{\partial\theta_{1}},\cdots,\frac{\partial}{\partial\theta_{m-1}}\right\}$
is the polar coordinates for $T_{q}N$. Denote by
$\{e_{1},e_{2},\ldots,e_{n}\}$ an orthonormal basis for
$T_{w}\Omega$. Then one can decompose $e_{i}$ as follows
\begin{eqnarray*}
e_{i}=a_{i}\cdot\mathrm{grad}\rho_{N}+b_{i}\cdot\frac{\partial}{\partial{s}}+\sum\limits_{j=1}^{m-1}c_{i}^{j}\cdot\frac{\partial}{\partial\theta_{j}},
\qquad i=1,2,\ldots,n,
\end{eqnarray*}
where $a_{i}$, $b_{i}$, $c_{i}^{j}$ are constants satisfying
\begin{eqnarray} \label{4-8}
a_{i}^{2}+b_{i}^{2}+\sum\limits_{j=1}^{m-1}\left(c_{i}^{j}\right)^{2}=1.
\end{eqnarray}
 By applying (\ref{3-2}) with $Q=N\times\mathbb{R}$ to the function $f$, it
follows that
\begin{eqnarray} \label{4-9}
\Delta
f=\left[\sum\limits_{i=1}^{n}\mathrm{Hess}_{N\times\mathbb{R}}\widetilde{\rho}(e_{i},e_{i})+\left\langle\mathrm{grad}_{N\times\mathbb{R}}\widetilde{\rho},H\right\rangle\right]_{\phi(w)},
\end{eqnarray}
where $H=\sum\limits_{i=1}^{n}\alpha(e_{i},e_{i})$ is the mean
curvature vector of $\phi(M)$ at the point $\phi(w)$ and the
orthonormal basis $\{e_{1},e_{2},\ldots,e_{n}\}$ of $T_{w}M$
identified with
$\{\phi_{\ast}(e_{1}),\phi_{\ast}(e_{2}),\ldots,\phi_{\ast}(e_{n})\}$.
By Theorem \ref{HCT} and (\ref{4-8}), we have
\begin{eqnarray*}
\sum\limits_{i=1}^{n}\mathrm{Hess}_{N\times\mathbb{R}}\widetilde{\rho}(e_{i},e_{i})&=&\sum\limits_{i=1}^{n}\mathrm{Hess}_{N}\rho_{N}(e_{i},e_{i})\\
&=&\sum\limits_{i=1}^{n}\sum\limits_{j=1}^{m-1}\left(c_{i}^{j}\right)^{2}\mathrm{Hess}_{N}\rho_{N}\left(\frac{\partial}{\partial\theta_{i}},\frac{\partial}{\partial\theta_{j}}\right)\\
&\geq&\sum\limits_{i=1}^{n}\left(1-a_{i}^{2}-b_{i}^{2}\right)\frac{C_{k}(r)}{S_{k}(r)}
\end{eqnarray*}
and
\begin{eqnarray*}
\left\langle\mathrm{grad}_{N\times\mathbb{R}}\widetilde{\rho},H\right\rangle=\left\langle\mathrm{grad}_{N}\rho_{N},H\right\rangle&=&
\left\langle\left(\mathrm{grad}_{N}\rho_{N}\right)^{\perp},H\right\rangle\\
&\leq&\|\left(\mathrm{grad}_{N}\rho_{N}\right)^{\perp}\|\cdot\|H\|=\|H\|\sqrt{1-\sum\limits_{i=1}^{n}a_{i}^{2}}\\
&\leq&h(x_{0},r)\sqrt{1-\sum\limits_{i=1}^{n}a_{i}^{2}}.
\end{eqnarray*}
Substituting the above two inequalities into (\ref{4-9}), together
with the fact $1-\sum\limits_{i=1}^{n}a_{i}^{2}\geq0$ and
$1-\sum\limits_{i=1}^{n}b_{i}^{2}\geq0$, yields
\begin{eqnarray}  \label{4-10}
\Delta f\geq(n-2)\frac{C_{k}(r)}{S_{k}(r)}-h(q,r)>0.
\end{eqnarray}
If $|h(x_{0},r)|<\digamma^{2}<\infty$, then we can choose
\begin{eqnarray*}
r\leq\left\{\mathrm{inj}_{N}(q),\frac{\pi}{2\sqrt{k}},(\frac{C_{k}}{S_{k}})^{-1}\cdot\frac{\digamma^{2}}{(n-2)}\right\}.
\end{eqnarray*}
If $\lim\limits_{r\rightarrow\infty}h(q,r_{0})=\infty$, there exists
$r_{0}$ such that
$(n-2)\frac{C_{k}(r_{0})}{S_{k}(r_{0})}-h(q,r_{0})=0$ since $h(q,r)$
is a continuous nondecreasing function in $r$. Then in this
situation, we can choose
\begin{eqnarray*}
r\leq\left\{\mathrm{inj}_{N}(q),\frac{\pi}{2\sqrt{k}},(\frac{C_{k}}{S_{k}})^{-1}\cdot\frac{h(q,r_{0})}{(n-2)}\right\}.
\end{eqnarray*}
Putting (\ref{4-10}) with $\mathrm{div}X=\Delta{f}$ into
(\ref{addd-1}) and (\ref{addd-2}), our estimates for
$\lambda_{1,\varphi}^{\ast}(\Omega)$ and $\lambda_{1,p}(\Omega)$ can
be obtained.
 \end{proof}

\begin{remark}
\rm{ If $\Omega$ is bounded and has the piecewise smooth boundary,
then putting (\ref{4-10}) with $\mathrm{div}X=\Delta{f}$ into
(\ref{addd-2}), the estimate (\ref{add-o-2}) follows. If $\Omega$ is
unbounded, one can choose an exhaustion
$\{\Omega_{i}\}_{i=1,2,3,\cdots}$ with
$\Omega_{i}\subset\Omega_{j}\subset\Omega$, $i<j$, and putting
(\ref{4-10}) into (\ref{addd-2}) for the bounded domain
$\Omega_{i}$, we have
\begin{eqnarray*}
\lambda_{1,p}(\Omega_{i})\geq\left[\frac{(n-2)\frac{C_{k}(r)}{S_{k}(r)}-h(q,r)}{p}\right]^{p},
\end{eqnarray*}
which implies the estimate (\ref{add-o-2}) by applying domain
monotonicity of the first eigenvalue of the $p$-Laplacian with
vanishing Dirichlet data and taking limits to both sides of the
above inequality as $i\rightarrow\infty$. Besides, clearly, when
$\varphi=const.$ or $p=2$, our estimates here are exactly the one in
\cite[Theorem 1.6]{bm2}.

}
\end{remark}

\section{Eigenvalue estimates for submanifolds with bounded $\varphi$-mean curvature in the hyperbolic space}
\renewcommand{\thesection}{\arabic{section}}
\renewcommand{\theequation}{\thesection.\arabic{equation}}
\setcounter{equation}{0} \setcounter{maintheorem}{0}

For an $n$-dimensional ($n\geq2$) submanifold $M$ of the weighted
manifold $\left(\mathbb{H}^{m}(-1),e^{-\varphi}dv\right)$, its
$\varphi$-mean curvature vector field $H_{\varphi}$ is given by
\begin{eqnarray*}
H_{\varphi}:=H+\left(\bar{\nabla}\varphi\right)^{\perp}
\end{eqnarray*}
where $\perp$ denotes the projection onto the normal bundle of  $M$,
$\bar{\nabla}$ is the gradient operator on the hyperbolic $m$-space
$\mathbb{H}^{m}(-1)$, and, as before, $H$ is the mean curvature
vector of $M$. We call $M$ is $\varphi$-\emph{minimal} if
$H_{\varphi}$ vanishes everywhere. See, e.g., \cite{lw,ww} for the
notion of $\varphi$-mean curvature and some interesting
applications.

\begin{remark}
\rm{ Clearly, if $\varphi=const.$, then $H_{\varphi}=H$, and in this
situation, ``\emph{minimal}" is equivalent to ``$\varphi$-minimal".
However, in general case, they are different. }
\end{remark}

Now, by applying the $\varphi$-minimal assumption and \cite[Theorem
1.3]{dm2}, we can prove the following result.

\begin{theorem} \label{theorem-5-main}
Let $M$ be an $n$-dimensional ($n\geq2$) complete noncompact
$\varphi$-minimal submanifold of the weighted manifold
$\left(\mathbb{H}^{m}(-1),e^{-\varphi}dv\right)$. If
$\sup_{M}\|\bar{\nabla}\varphi\|<n-1$, then
\begin{eqnarray}  \label{add-6-1}
\lambda_{1,p}(M)\geq
\left(\frac{n-1-\sup_{M}\|\bar{\nabla}\varphi\|}{p}\right)^{p}>0.
\end{eqnarray}
\end{theorem}

\begin{proof}
By a direct calculation, we have
 \begin{eqnarray*}
\sup_{M}\|H\|\leq\sup_{M}\sqrt{\|H\|^{2}+\|(\bar{\nabla}\varphi)^{\top}\|^{2}}=\sup_{M}\|H+(\bar{\nabla}\varphi)^{\top}\|=\sup_{M}\|H_{\varphi}-\bar{\nabla}\varphi\|,
\end{eqnarray*}
where $\top$ denotes the projection onto the tangent bundle of $M$.
Therefore, if $M$ is $\varphi$-minimal and
$\sup_{M}\|\bar{\nabla}\varphi\|<n-1$, then $\sup_{M}\|H\|<n-1$. By
applying \cite[Theorem 1.3]{dm2} directly, we have
\begin{eqnarray*}
\lambda_{1,p}(M)\geq \left(\frac{n-1-\sup_{M}\|H\|}{p}\right)^{p}>0.
\end{eqnarray*}
This implies
\begin{eqnarray*}
\lambda_{1,p}(M)&\geq&
\left[\frac{n-1-\sup_{M}\sqrt{\|H\|^{2}+\|(\bar{\nabla}\varphi)^{\top}\|^{2}}}{p}\right]^{p}\\
&=&\left(\frac{n-1-\sup_{M}\|H_{\varphi}-\bar{\nabla}\varphi\|}{p}\right)^{p}\\
&=&\left(\frac{n-1-\sup_{M}\|\bar{\nabla}\varphi\|}{p}\right)^{p}>0
\end{eqnarray*}
provided $M$ is $\varphi$-minimal and
$\sup_{M}\|\bar{\nabla}\varphi\|<n-1$.
\end{proof}

\begin{remark}
\rm{Clearly, when $\varphi=const.$, our estimate (\ref{add-6-1})
becomes
\begin{eqnarray*}
\lambda_{1,p}(M)\geq \left(\frac{n-1}{p}\right)^{p}>0,
\end{eqnarray*}
which is exactly (1.5) of \cite{dm2}. When $\varphi=const.$ and
$p=2$, our Theorem \ref{theorem-5-main} degenerates into
\cite[Corollary 3]{cl}. }
\end{remark}

\section{Lower bounds for the first Dirichlet eigenvalues of the weighted Laplacian and the $p$-Laplacian on geodesic balls}
\renewcommand{\thesection}{\arabic{section}}
\renewcommand{\theequation}{\thesection.\arabic{equation}}
\setcounter{equation}{0} \setcounter{maintheorem}{0}

For an $n$-dimensional ($n\geq2$) complete manifold $M$ with
sectional curvature bounded from above by some constant $k$, Cheng
\cite{cy} proved
$\lambda_{1}(B_{M}(q,r))\geq\lambda_{1}(B_{\mathcal{M}(n,k)}(r))$
with equality holds if and only if $B_{M}(q,r)$ is isometric to
$B_{\mathcal{M}(n,k)}(r)$, where $B_{M}(q,r)$ is the geodesic ball,
with center $q\in M$ and radius $r$, within the cut-locus of $q$,
$B_{\mathcal{M}(n,k)}(r)$ is the geodesic ball of radius $r$ in the
$n$-dimensional space form $\mathcal{M}(n,k)$ with constant
sectional curvature $k$. By using the radial sectional curvature
(whose upper bound is given by a continuous function of the
Riemannian distance parameter) assumption and spherically symmetric
manifolds as model spaces, Freitas, Mao and Salavessa \cite[Theorem
4.4]{fmi} improved Cheng's conclusion mentioned above a lot. The
advantage of Freitas-Mao-Salavessa's theory has been shown
intuitively by numerically calculating the first Dirichlet
eigenvalue of the Laplacian on torus, elliptic paraboloid and saddle
(see \cite[Section 6]{fmi}). Besides, the principle of doing
numerical calculation for the first Dirichlet eigenvalue of the
Laplacian on parameterized surfaces has been given in \cite{hm,m1}.

It is well-known that the first Dirichlet eigenvalue
$\lambda_{1}(B_{\mathbb{R}^{n}}(r))$ of the Laplacian of a ball in
$\mathbb{R}^{n}$ with radius $r$ is
$\lambda_{1}(B_{\mathbb{R}^{n}}(r))=\left(\frac{J_{\frac{n}{2}-1}}{r}\right)^{2}$,
where $J_{\frac{n}{2}-1}$ is the first zero point of the
$\left(\frac{n}{2}-1\right)$-st Bessel function. By Cheng's
eigenvalue comparison \cite{cy} (or its generalization \cite[Theorem
4.4]{fmi}), for an $n$-dimensional ($n\geq2$) complete Riemannian
manifold $M$ with non-positive sectional curvature, one has
\begin{eqnarray} \label{F-1}
\lambda_{1}(B_{M}(q,r))\geq\left(\frac{J_{\frac{n}{2}-1}}{r}\right)^{2},
\end{eqnarray}
where the geodesic ball $B_{M}(q,r)$ is within the cut-locus of
$q\in M$. The equality in (\ref{F-1}) holds if and only if
$B_{M}(q,r)$ is isometric to $B_{\mathbb{R}^{n}}(r)$.

However, applying Lemma \ref{lemma2-3}, we can prove the following
sharper lower bounds.

\begin{theorem} \label{theorem6-1}
Let $M$ be an $n$-dimensional ($n\geq2$) complete manifold and a
point $q\in M$. Let $B_{M}(q,r)$ be a geodesic ball with center
$q\in M$ and radius $r$, where $r<\mathrm{inj(q)}$ with
$\mathrm{inj(q)}$ the injective radius of $q$. Let
$\kappa(q,r)=\sup\{K_{M}(x)|x\in{B_{M}(q,r)}\}$, where $K_{M}(x)$
are sectional curvatures of $M$ at $x$. Assume that $\varphi$ is a
real-valued smooth function on $M$ with $\|\nabla\varphi\|\leq
c^{+}$, where $c^{+}$ is the supremum of the norm of the gradient of
$\varphi$. Then for $k>0$, we have
\begin{eqnarray*}
\lambda_{1,\varphi}(B_{M}(q,r))\geq\left\{
\begin{array}{lll}
\frac{1}{4}\cdot\left[(n-1)k\coth(kr)+\frac{1}{r}-c^{+}\right]^{2},~~~~~\qquad\mathrm{if}~~\kappa(q,r)=-k^{2},\\[2mm]
\left(\frac{n}{2r}-\frac{c^{+}}{2}\right)^{2},~~~~~~~~~~~~~~\qquad\qquad\qquad\qquad\mathrm{if}~~\kappa(q,r)=0~\mathrm{and}~\lambda_{1,\varphi}(M)>0,\\[2mm]
\left[\frac{(n-1)kr\cot(kr)+1}{2r}-\frac{c^{+}}{2}\right]^{2},~~~~~~~~~~\qquad\qquad~\mathrm{if}~~\kappa(q,r)=k^{2}~\mathrm{and}~r<\frac{\pi}{2k}
&\quad
\end{array}
\right.
\end{eqnarray*}
and
\begin{eqnarray*}
\lambda_{1,p}(B_{M}(q,r))\geq\left\{
\begin{array}{lll}
\left(\frac{1}{p}\right)^{p}\cdot\left[(n-1)k\coth(kr)+\frac{1}{r}\right]^{p},~~~~~\quad\mathrm{if}~~\kappa(q,r)=-k^{2},\\[2mm]
\left(\frac{n}{pr}\right)^{p},~~~~~~~~~\qquad\qquad\qquad\qquad\qquad\mathrm{if}~~\kappa(q,r)=0~\mathrm{and}~~\lambda_{1,p}(M)>0,\\[2mm]
\left[\frac{(n-1)kr\cot(kr)+1}{pr}\right]^{p},~~~~~\qquad\qquad\qquad~\mathrm{if}~~\kappa(q,r)=k^{2}~\mathrm{and}~r<\frac{\pi}{2k},
&\quad
\end{array}
\right.
\end{eqnarray*}
where $c^{+}$ satisfies
\begin{eqnarray*}
c^{+}<\left\{
\begin{array}{lll}
(n-1)k\coth(kr)+\frac{1}{r},~~~~~\qquad\qquad\mathrm{if}~~\kappa(q,r)=-k^{2},\\[2mm]
\frac{n}{r},~~~~~~~~~~~~\qquad\qquad\qquad\qquad\qquad\mathrm{if}~~\kappa(q,r)=0~\mathrm{and}~\lambda_{1,\varphi}(M)>0,\\[2mm]
\frac{(n-1)kr\cot(kr)+1}{r},~~~~~~\qquad\qquad\qquad~\mathrm{if}~~\kappa(q,r)=k^{2}~\mathrm{and}~r<\frac{\pi}{2k}.
&\quad
\end{array}
\right.
\end{eqnarray*}
\end{theorem}

\begin{proof}
As before, let $\nabla$ and $\Delta$ be the gradient and the Laplace
operators on $M$ respectively. Choose $X=\nabla\rho^{2}$ with
$\rho(x)=\mathrm{dist}_{M}(q,x)$. Then
$\|X\|=2\rho\|\nabla\rho\|=2\rho$. By (\ref{HCTLI}), we have
\begin{eqnarray*}
\mathrm{div}X=\Delta\rho^{2}\geq2(n-1)\rho\cdot\frac{1}{\rho}+2=2n,
\quad\qquad &&\mathrm{if}~~\kappa(q,r)=0, \\
\mathrm{div}X=\Delta\rho^{2}\geq2(n-1)kr\cot(kr)+2, \quad
&&\mathrm{if}~~\kappa(q,r)=k^{2},~r<\frac{\pi}{2k},
\end{eqnarray*}
and
\begin{eqnarray*}
\mathrm{div}X=\Delta\rho^{2}\geq2(n-1)kr\coth(kr)+2, ~~\quad
\mathrm{if}~~\kappa(q,r)=-k^{2},
\end{eqnarray*}
 which implies
\begin{eqnarray*}
c(B_{M}(q,r))\geq\frac{n}{r}, \qquad &&\mathrm{if}~\kappa(q,r)=0,\\
c(B_{M}(q,r))\geq\frac{(n-1)kr\cot(kr)+1}{r}, \qquad
&&\mathrm{if}~\kappa(q,r)=k^{2},~r<\frac{\pi}{2k},
\end{eqnarray*}
and
\begin{eqnarray*}
c(B_{M}(q,r))\geq\frac{(n-1)kr\coth(kr)+1}{r}, \qquad
\mathrm{if}~\kappa(q,r)=-k^{2}.
\end{eqnarray*}
By applying Lemma \ref{lemma2-3}, one can obtain estimates in
Theorem \ref{theorem6-1}. However, as pointed out in Remark
\ref{remark2-4}, in order to use estimates in Lemma \ref{lemma2-3},
one has to show $\int_{B_{M}(q,r)}{\rm{div}}(|f|^{p}X)=0$ or
$\int_{B_{M}(q,r)}{\rm{div}}(f^{2}Xe^{-\varphi})=0$ for all $f\in
C^{\infty}_{0}(B_{M}(q,r))$ and the chosen vector filed $X$ which is
smooth almost every where in $B_{M}(q,r)$. This fact can be easily
proven through replacing $\mathrm{div}(f^{2}X)$ by
$\mathrm{div}(|f|^{2}Xe^{-\varphi})$ or $\mathrm{div}(|f|^{p}X)$ in
the last part of the proof of \cite[Theorem 4.1]{bm}.
\end{proof}

\begin{remark}
\rm{If $\kappa(q,r)=-k^{2}$ or $\kappa(q,r)=0$, then
$\mathrm{inj}(q)=\infty$, which implies that $M$ is noncompact. For
the case of $\kappa(q,r)=-k^{2}$, letting $r\rightarrow\infty$, then
$B_{M}(B(q,r))$ tends to $M$, and
$\lambda_{1,\varphi}(M)\geq\left[\frac{(n-1)k-c^{+}}{2}\right]^{2}$
and $\lambda_{1,p}(M)\geq\left[\frac{(n-1)k}{p}\right]^{p}$, which
are exactly the estimates given in Lemma \ref{lemma2-8}. If
$\varphi=const.$ (or $p=2$) and $M$ has non-positive sectional
curvature (which satisfies assumption $\kappa(q,r)=0$), then
$\lambda_{1,\varphi}(B_{M}(q,r))=\lambda_{1}(B_{M}(q,r))$ (or
$\lambda_{1,p}(B_{M}(q,r))=\lambda_{1}(B_{M}(q,r)))$ and by Theorem
\ref{theorem6-1}, one has
$\lambda_{1}(B_{M}(q,r))\geq\frac{n^{2}}{4r^{2}}$, which is not so
good as the estimate (\ref{F-1}), since
$J_{\frac{n}{2}-1}>\frac{n}{2}$ for $n\in\mathbb{N}_{+}$ and
$n\geq2$. However, this lower bound becomes more and more sharper as
$r$ increases, since $2J_{\frac{n}{2}-1}/n\rightarrow1$ as
$n\rightarrow\infty$. }
\end{remark}

\section*{Acknowledgments}

This work was partially supported by the NSF of China (Grant Nos.
11401131 and 11801496), China Scholarship Council, the Fok Ying-Tung
Education Foundation (China), and Key Laboratory of Applied
Mathematics of Hubei Province (Hubei University).

 \end{document}